\documentclass[a4paper,12pt,reqno]{amsart}
\usepackage{amsfonts,amssymb,hyperref,amsthm,enumerate,
	color,stmaryrd,pgf,tikz,comment,multirow}
\usepackage[left=2.6cm,right=2.6cm,top=3cm,bottom=3cm,bindingoffset=0cm]
{geometry}
\usepackage{array}   \usepackage{multirow}  \usepackage{booktabs,caption,fixltx2e}
\usepackage[flushleft]{threeparttable}

\usepackage{booktabs}
\newtheorem{theorem}{Theorem}[section]
\newtheorem{lemma}[theorem]{Lemma}
\newtheorem{proposition}[theorem]{Proposition}

\theoremstyle{remark}

\theoremstyle{definition}

\newtheorem{corollary}[theorem]{Corollary}
\newtheorem{problem}{Problem}

\newtheorem*{case}{Case}

\DeclareMathOperator{\Aut}{Aut}

\DeclareMathOperator{\Syl}{Syl}
\DeclareMathOperator{\GF}{GF}
\DeclareMathOperator{\GL}{GL}
\DeclareMathOperator{\ch}{~char~}
\DeclareMathOperator{\rank}{rank}
\DeclareMathOperator{\rad}{rad}

\title{A characterization of nilpotent bicyclic groups}
\author{Kan Hu}
\address{K. Hu
\newline\indent
Department of Mathematics, Zhejiang Ocean University, Zhoushan, Zhejiang 316022, P.R. China}
\email{hukan@zjou.edu.cn}
\thanks{This work was supported by National Natural
Science Foundation of China (12471332).
}
\begin{document}
	\maketitle
\begin{abstract}
A group is called $(m,n)$-bicyclic if it can be expressed as a product of two cyclic subgroups of
orders $m$ and $n$, respectively. The classification and characterization of finite bicyclic groups have long
been important problems in group theory, with applications extending to symmetric embeddings of
the complete bipartite graphs.
A classical result by Douglas establishes that every bicyclic group is supersolvable.
More recently, Fan and Li (2018) proved that  every finite $(m,n)$-bicyclic group is abelian
if and only if $\gcd(m,\phi(n))=\gcd(n,\phi(m))=1$, where $\phi$ is Euler's totient function.
In this paper we generalize this result further and show that every $(m,n)$-bicyclic group is nilpotent
if and only if $\gcd(n,\phi(\mathrm{rad}(m)))=\gcd(m,\phi(\mathrm{rad}(n)))=1$, where
 $\mathrm{rad}(m)$ denotes the radical of $m$ (the product of its distinct prime divisors).
\end{abstract}

\section{Introduction}
Throughout the paper, all groups  considered are finite unless stated otherwise.
A group $G$ is \textit{factorisible} if $G=AB$ for some subgroups $A$ and $B$ of $G$.
In particular, the factorization $G=AB$ is \textit{exact} if $A\cap B=1$. The classical
 factorization problem in group theory, first posed by Ore in 1937, is to 
 describe and classify the groups $G$ that admit a factorization $G=AB$ 
 for two given groups $A$ and $B$~\cite{Ore1937}.
 
This problem has been extensively studied at the level of structural description. 
For example, using elementary techniques, It\^{o} showed if both $A$ and $B$ are abelian, 
then $G$ is metabelian~\cite{Ito1955}. Douglas established that if both $A$ and $B$ are cyclic,
then $G$ is supersolvable~\cite{Douglas1961}.
One of the most well-known results in this direction is due to Kegel and Wielandt, 
who proved that if both $A$ and $B$ are nilpotent, then $G$ is solvable~\cite{Kegel1961,Wielandt1958}.

 In contrast, at the classification level, the factorization problem becomes significantly more
 challenging. A particular longstanding open  problem is the  classification of finite bicyclic groups, that is, 
   finite  groups which can be expressed as a product of two cyclic subgroups.
 As noted earlier, every bicyclic group is supersolvable, and hence possesses a Sylow tower. 
 It is therefore  natural to begin the classification by considering bicyclic groups that are finite $p$-groups.  
 Huppert proved that a bicyclic $p$-group is metacyclic when $p$ is an odd prime~\cite{Huppert1953}. 
 However, it is well known that  non-metacyclic bicyclic $2$-groups exist. A  classification of 
 metayclic $p$-groups was obtained
   by Xu for $p>2$, and by Xu and Zhang for $p=2$. The classification of bicyclic $2$-groups was eventually
   completed by Janko~\cite{Janko2008}. 
   
   In parallel, bicyclic groups have also been studied by combinatorialists in connection with symmetric embeddings
   of graphs into oriented closed surfaces.  A $2$-cell embedding of a graph $\Gamma$ into an oriented closed surface $\mathbb{S}$
   is called a \textit{map} on $\mathbb{S}$, and denoted by $M$. In particular, if the underlying graph $\Gamma$ is bipartite, 
   the associated map is often referred to as a \textit{dessin} and denoted by $D$, following Grothendieck's 
   formulation of the theory of dessins d'enfants. An \textit{automorphism}
   of a map is an automorphism of the underlying graph which extends to an orientation-preserving 
   self-homeomorphism of the supporting surface. It is well known that the automorphism group $\Aut^+(M)$ of a map $M$
   acts semi-regularly on the set of arcs of the underlying graphs. If this action is also transitive, and so regular,
   then the map $M$ is called \textit{regular}. Similarly,  for a dessin $D$, the subgroup $\Aut^+_0(M)$,
   consisting of color-preserving automorphisms, acts semi-regularly on the set of edges of $D$. If this action is regular, then 
    $D$ is called a \textit{regular dessin}.

    A dessin is referred to as a \textit{complete dessin} if its underling graph
   is a complete bipartite graph; or more specifically, it is called \textit{$(m,n)$-complete} if the underlying graph is 
   the complete bipartite graph $K_{m,n}$. We note that a complete regular dessin $D$ with underlying 
   graph $K_{m,n}$ gives rise to a regular map if and only if $m=n$ and $D$ admits
   an orientation-preserving automorphism that swaps the vertices of different colors.
   
   An important problem in topological graph theory is the classification of complete regular dessins. 
   This problem is closely related to bicyclic  groups, as shown in~\cite{JNS2007}: if $D$ is a 
   complete regular dessin with underlying graph $K_{m,n}$, then
   $G:=\Aut_0^+(D)$ admits an exact factorization $G=\langle a\rangle \langle b\rangle $ of two cyclic subgroups 
   $\langle a\rangle $ and $\langle b\rangle $ of orders $m$ and $n$, respectively. Moreover, the dessin $D$ is 
   determined (up to isomorphism) by the triple $(G,a,b)$    in the sense that two complete regular dessins $D$ and $D'$, 
   corresponding to the triples $(G,a,b)$ and $(G',a',b')$, are isomorphic
   if and only if  there is a group isomorphism from $G$ onto $G'$ mapping $a\mapsto a'$ 
   and $b\mapsto b'$. This correspondence makes the classification of bicyclic groups a fundamental step towards
   the classification of complete regular dessins. Using this connection, 
   regular maps with underlying    graphs $K_{n,n}$ have been successfully classified in a series of works~\cite{DJKNS2007,DJKNS2010,Jones2010,JNS2008,JNS2007}. Additionally, 
  partial classifications of complete regular dessins have been obtained in~\cite{CH2024, Fan2022,FHNSW2020,FL2018,FLQ2023,FLQ2018,HNW2019}.
 
Among the various results, one particularly interesting case stands out. It was shown that, up
to isomorphism, there exists a \textit{unique} regular map with underlying graph $K_{n,n}$ if and only if $\gcd(n,\phi(n))=1$~\cite{JNS2007}. More generally, there exists a \textit{unique} regular dessin with underlying graph 
$K_{m,n}$ if and only if the pair $(m,n)$ 
is \textit{singular}, meaning $\gcd(m,\phi(n))=1$ and $\gcd(\phi(m),n)=1$~\cite{FL2018}. Reformulated in
 group-theoretic language, this result is equivalent to the following statement: 
\textit{Every} $(m,n)$-bicyclic group is abelian if and only if $(m,n)$ is a singular pair. 

We now mention a few classical results in group theory that exhibit similar structural properties. 
A natural number $n$ is said to be \textit{cyclic} (respectively, \textit{abelian}, \textit{nilpotent}, \textit{supersolvable}, or \textit{solvable}) 
if every finite group of order $n$ is cyclic 
(respectively, abelian, nilpotent, supersolvable, or solvable). To state these results, we define a function 
$\psi$ in the following way: 
\begin{enumerate}[(a)]
\item $\psi(1)=1$,
\item for a prime $p$ and an integer $e\geq1$, $\psi(p^e)=\prod_{i=1}^e(p^i-1)$,
\item $\psi$ is multiplicative; that is, if $\gcd(m,n)=1$, then $\psi(mn)=\psi(m)\psi(n)$.
\end{enumerate}
Thus, for $n=\prod_{i=1}^rp_i^{e_i}$, we have 
\[
\psi(n)=\prod_{i=1}^r\prod_{j=1}^{e_i}(p^j-1).
\]
With this definition, it turns out that
\begin{enumerate}[(a)]
\item  $n$ is nilpotent if and only if $\gcd(n,\psi(n))=1$,
\item $n$ is abelian if and only if $\gcd(n,\psi(n))=1$ and $n$ is cube-free, 
\item $n$ is cyclic if and only if  $\gcd(n,\psi(n))=1$ 
and $n$ is square-free (or equivalently, $\gcd(n,\phi(n))=1$). 
\end{enumerate}
Related results on supersolvable and solvable numbers
can be found in~\cite{Bray1982, Pazderski1959, PS2000}.

 Since  bicyclic groups are, in general, supersolvable, and since the class of nilpotent groups 
 lies strictly between the classes of abelian and supersolvable group, it is natural
 to consider the following problem: 
\begin{problem}
Determine necessary and sufficient conditions on $m$ and $n$ under which every $(m,n)$-bicyclic group is nilpotent.
\end{problem}   
The following is the main result of this paper, which provides a complete solution to the problem.
 \begin{theorem}\label{main}
Every $(m,n)$-bicyclic group is nilpotent if and only if
\begin{align}\label{Mcond}
\gcd\big(n,\phi(\rad(m))\big)=\gcd\big(m,\phi(\rad(n))\big)=1,
\end{align}
 where
 $\mathrm{rad}(m)$ denotes the radical of $m$, i.e., the product of its distinct prime divisors of $m$.
\end{theorem}  
After presenting the necessary group-theoretic preliminaries in the next section,  we will
prove Theorem~\ref{main} in Section~3. The final section is devoted to discussing a few related open problems
and potential directions for further research.

\section{Preliminaries}
In this section we collect the necessary notation and fundamental results from
group theory that will be used throughout the remainder of the paper.

Let $G$ be a finite group, we denote its derived subgroup by $G'$, its center by $Z(G)$, and its Frattini subgroup by 
$\Phi(G)$. For a prime divisor $p$ of $|G|$, we write $\Syl_p(G)$ for the set of all Sylow $p$-subgroups of $G$.

A group $G$ is \textit{nilpotent} if it admits a central series of finite length. 
It is a standard result that a finite  group $G$ is nilpotent if and only if all of its Sylow subgroups are normal;
 in this case, $G$ is the (internal) direct product of its Sylow subgroups. 

 The following characterization of nilpotent groups will be useful in our later arguments.
\begin{proposition}[{\rm \cite[III, 2.3]{Huppert1967}}]\label{Nilp}
A finite group $G$ is nilpotent if and only if for any two elements $x,y\in G$ with $\gcd(|x|,|y|)=1$ we have $[x,y]=1$.
\end{proposition}
In particular, every finite $p$-group is nilpotent. Finite $p$-groups possess the following basic properties:
\begin{proposition}[{\rm \cite[III, 3.14, 3.15, 3.17]{Huppert1967}}]\label{Basis}
Let $G$ be a finite $p$-group with $|G/\Phi(G)|=p^d$. Then the following statements hold:
\begin{enumerate}[\rm(a)]
\item $\Phi(G)=G'\mho_1(G)$ where $\mho_1(G)=\langle g^p\mid g\in G\rangle$.
\item $G/\Phi(G)$ is an elementary abelian $p$-group.
\item Every minimal generating set of $G$ is of cardinality $d$, which is called the rank of $G$.
\item For each $\sigma\in\Aut(G)$, the map $\bar\sigma$ defined by $(g\Phi(G))^{\bar\sigma}=g^\sigma\Phi(G)$
is an automorphism of $G/\Phi(G).$ Moreover, the map $\sigma\mapsto\bar\sigma$ defines
a group homomorphism $\Aut(G)\to\Aut(G/\Phi(G))$ whose kernel is a finite $p$-group.
\end{enumerate}
\end{proposition}

A group $G$ is \textit{solvable} if it has a finite sequence of normal subgroups 
\[
1=N_0<N_1<\cdots<N_s=G
\]
 such that each factor $N_{i}/N_{i-1}$ is elementary abelian of prime power order.
In particular, if all the factors are cyclic of prime order, then $G$ is termed \textit{supersolvable}.
Moreover, let $G$ be a finite group of order $|G|=\prod_{i=1}^sp_i^{e_i}$, where the primes
are ordered such that $p_1>p_2>\cdots>p_s$. We say that $G$ has a \textit{Sylow tower} if  
there exists a sequence of normal subgroups 
\[
1=N_0<N_1<\cdots<N_s=G
\]
 such that for each $i=1,2,\ldots,s$, the order of $N_i$ is  $\prod_{j=1}^ip_j^{e_j}$.

\begin{proposition}[{\cite[VI 9.1]{Huppert1967}}]\label{Super}
Every finite supersolvable group $G$ has a Sylow tower. In particular,
 if $p$ is the largest prime factor of $|G|$, then the Sylow $p$-group of $G$ is normal.
\end{proposition}

An important subclass of supersolvable groups is the class of metacyclic groups. Recall that
a group $G$ is  \textit{metacyclic} if it is an extension of a cyclic group by another cyclic group. In 
other words, $G$ contains a cyclic normal subgroup $A$ such that the quotient $G/A$
is also cyclic. One of the key results about metacyclic groups is H\"older's theorem, which is fundamental in this context.
\begin{proposition}[{\rm \cite[I, 14.8]{Huppert1967}}]\label{Holder}
If $G$ is an extension of a finite cyclic group $A$ of order $m$ by a finite cyclic group $B$ of order $n$, then
$G$ has a presentation of the form
\[
G=\langle a,b\mid a^m=1, b^n=a^s, a^b=a^r\rangle,
\]
where $r,s\in\mathbb{Z}_m$ satisfy the following congruences:
\[
(r-1)s\equiv0\pmod{m}\quad\text{and}\quad r^n\equiv1\pmod{m}.
\]
Conversely, a group $G$ defined by the presentation above and the numerical conditions above
determines an extension of a cyclic subgroup of order $m$ by a cyclic group of order $n$.
\end{proposition}

More generally, a finite group $G$ is termed \textit{bicyclic}, or more specifically \textit{$(m,n)$-bicyclic},
if it admits a factorization
$G=\langle a\rangle\langle b\rangle$ of two cyclic subgroups $\langle a\rangle$ and $\langle b\rangle$, 
where $|a|=m$ and $|b|=n$.
\begin{proposition}[{\rm \cite[III, 11.5; VI, 10.1]{Huppert1967}}]\label{Bicyclic}
Every bicyclic group is supersolvable. Moreover, if $p>2$ then every bicyclic $p$-group is metacyclic.
\end{proposition}
\begin{proposition}[{\rm \cite[VI, 4.7]{Huppert1967}}]\label{Prod}
Let $G=AB$ be a factorization of a group $G$. For each prime $p$, 
there exist Sylow $p$-subgroups $P\in\Syl_p(G)$, $P_1\in\Syl_p(A)$ and $P_2\in\Syl_p(B)$ 
such that $P=P_1P_2$.
\end{proposition}

A subgroup $H$ of $G$ is called a \textit{Hall subgroup} if $\gcd(|H|,|G:H|)=1$.
The following Schur-Zassenhauss theorem is well known.
\begin{proposition}[{\rm  \cite[I, 18.2]{Huppert1967}}]\label{Zass}
Let $G$ be a finite group. If $H$ is a normal Hall subgroup of $G$, then $H$
has a complement in $G$ and all complements of $H$ are pairwise conjugate in $G$.
\end{proposition}

\section{Proof of Theorem~\ref{main}}
Before proving Theorem~\ref{main}, we first present several technical lemmas that will be useful for the proof.
\begin{lemma}\label{Rem}
Two positive integers $m$ and $n$ do not satisfy the condition \eqref{Mcond} if and only if $m$ and $n$
satisfy one of the following conditions: 
 \begin{enumerate}[\rm(a)]
 \item both $m$ and $n$ are odd, and there exist a pair of primes $p\neq q$ 
 with $p|m$ and $q|n$  such that either $p| (q-1)$ or $q|( p-1)$,
 \item both  $m$ and $n$ are even, but not both equal to powers of the common prime $2$,
 \item $m$ is even and $n>1$ is odd, or $n$ is even and $m>1$ is odd.
   \end{enumerate}
\end{lemma}
\begin{proof}
Suppose first that $m$ and $n$ satisfy one of the conditions (a)--(c). We proceed to show that  either 
$\gcd(m,\phi(\rad(n))\neq1$ or $\gcd(n,\phi(\rad(m))\neq1$. 

   In case (a), both $p$ and $q$ are odd primes. If $p|(q-1)$ then $p|\gcd(m,\phi(\rad(n))$; if $q|(p-1)$ then $q|\gcd(n,\phi(\rad(m))$.
   
   In case (b), we assume $m=2^em_1$ and $n=2^fn_1$ where $e\geq1$, $f\geq1$, and both $m_1$ and $n_1$ are odd but 
    not both equal to $1$. If $m_1=1$, then $n_1>1$, and so $2|\gcd(m,\phi(\rad(n))$. If $m_1>1$, 
then $2|\gcd(m,\phi(\rad(n)))$. 

In case (c), if $m$ is even and $n>1$ is odd, then $2\mid\gcd(m,\phi(\rad(n))$;
if $m>1$ is odd and $n$ is even, then $2\mid\gcd(n,\phi(\rad(m))$. 

Now suppose that $m$ and $n$ satisfy none of the conditions (a)--(c). 

If both $m$ and $n$ are odd, then either one of $m$ and $n$ is $1$, or  they
are powers of a common odd prime, or they contains at lease two distinct prime divisors, and
for any pair of primes $p\neq q$ with $p|m$ and $q|n$,  we have $p\nmid (q-1)$ and $q\nmid( p-1).$
In this case, it is evident that the condition \eqref{Mcond} holds. 

If both $m$ and $n$ are even, then they must be powers of $2$, so the condition \eqref{Mcond} holds. 

If $m$ is even and $n$ is odd, then $n=1$; if $m$ is odd and $n$ is even, then $m=1$. In both cases,
the condition \eqref{Mcond}  holds as well.

Thus, we conclude the proof.
\end{proof}

\begin{corollary}\label{Rem1}
Two positive integers $m$ and $n$ satisfy the condition \eqref{Mcond} if and only if they satisfy one of the following conditions:
 \begin{enumerate}[\rm(a)]
 \item $m=1$ or $n=1$,
 \item $m=p^e$ and $n=p^f$ for some prime $p$ and integers $e,f\geq1$,
  \item Both $m$ and $n$ are odd numbers, and for each pair of primes $p\neq q$ with $p|m$ and $q|n$, 
  we have $p\nmid (q-1)$ and $q\nmid( p-1).$
 \end{enumerate}
\end{corollary}

\begin{lemma}\label{Odd}
Suppose that both $m$ and $n$ are odd, and there exist distinct primes $p\neq q$ 
with $p|m$ and $q|n$  such that either $p| (q-1)$ or $q|( p-1).$
Then there exists a non-nilpotent $(m,n)$-bicyclic group.
\end{lemma}
\begin{proof}
By symmetry, it suffices to consider the case where $p|(q-1)$. Write 
$m=p^em_1$ and $n=q^fn_1$, where $e,f\geq1$, $p\nmid m_1$, and $q\nmid n_1$.

Since $p|(q-1)$, we have $p|\phi(q^f)=q^{f-1}(q-1)$, so $\mathbb{Z}_{q^f}^*$ contains an element
of order $p$. Hence, there exists an integer $r\in\mathbb{Z}_n^*$ such that 
\[
r^p\equiv1\pmod{q^f}\quad\text{and}\quad r\not\equiv1\pmod{q^f}.
\]
Let us define a group $G$ by the presentation
\[
G=\langle a,b\mid a^m=b^n=1, b^a=b^r\rangle.
\]
By H\"older's theorem (Proposition~\ref{Holder}), this defines a well-defined metacyclic group of order $mn$, and $G$ is an
$(m,n)$-bicyclic group. 

Now define $a_1=a^{m_1}$ and $b_1=b^{n_1}$.
Then $|a_1|=p^e$, $|b_1|=q^f$, and hence $\gcd(|a_1|,|b_1|)=1$. From the defining relation $b^a=b^r$, 
 we obtain
 \[
 b_1^{a_1}=(b^{n_1})^{a^{m_1}}=(b^{a^{m_1}})^{n_1}=b_1^{r^{m_1}}.
 \]
 Since $p\nmid m_1$ and $r\not\equiv1\pmod{q^f}$, it follows that $r^{m_1}\not\equiv1\pmod{q^f}$, and thus $[a_1,b_1]\neq1$. 
Therefore, $a_1$ and $b_1$ are two elements of coprime orders in $G$ that do not
commute. By Proposition~\ref{Nilp}, $G$ is not nilpotent.
\end{proof}

\begin{lemma}\label{Even}
Suppose that  both  $m$ and $n$ are even, but not both equal to powers of the common prime $2$.
Then there is a non-nilpotent $(m,n)$-bicyclic group.
\end{lemma}
\begin{proof}
Write $m=2^em_1$ and $n=2^fn_1$, where $e,f\geq1$,  
and both $m_1$ and $n_1$ are odd. Since $m$ and $n$ are not both powers of $2$, we
may assume without loss of generality that $m_1>1$ (otherwise $n_1>1$, and the argument is similar).

Since $m_1>1$ is odd, it follows that $2|\phi(m_1)$. Hence, the multiplicative group $\mathbb{Z}_m^*$ 
of units modulo $m$ has even order $\phi(m)=\phi(2^e)\phi(m_1)$, and so it contains an element $r\in\mathbb{Z}_m^*$ such that
 \[
 r^2\equiv1\pmod{m_1}\quad\text{but}\quad r\not\equiv1\pmod{m_1}.
 \]
Let us define a group $G$ by the presentation
\[
G=\langle a,b\mid a^m=b^n=1, a^b=a^r\rangle.
\]
By H\"older's theorem, this gives a well-defined  metacyclic group of order $mn$.

Now set $a_1=a^{2^e}$ and $b_1=b^{n_1}$. Then 
 $| a_1|=m_1$, $|b_1|=2^f$, and $\gcd(|a_1|,|b_1|)=1$. From the defining relation $a^b=a^r$, we compute:
  \[
  a_1^{b_1}=(a^{2^e})^{b^{n_1}}=(a^{b^{n_1}})^{2^e}=(a^{r^{n_1}})^{2^e}=a_1^{r^{n_1}}.
  \]
 Since $n_1$ is odd and $r\not\equiv1\pmod{m_1}$, we have
$r^{n_1}\not\equiv1\pmod{m_1}$, so $[a_1,b_1]\neq1$.

Thus, $a_1$ and $b_1$ are elements  of coprime orders in $G$ which 
do not commute. By Proposition~\ref{Nilp}, the group $G$ is non-nilpotent.
\end{proof}

\begin{lemma}\label{OddEven}
Suppose that $m$ is even and $n>1$ is odd, or  $n$ is even and $m>1$ is odd. Then there
exists a non-nilpotent $(m,n)$-bicyclic group.
\end{lemma}
\begin{proof}
Assume first that $m$ is even, and $n>1$ is odd. Write $m=2^em_1$, where $e\geq1$ and $m_1$ is odd. 
Since $n>1$ is odd, the multiplicative group $\mathbb{Z}_n^*$
has even order $\phi(n)$, and so contains an element $r\in\mathbb{Z}_n^*$ with
\[
r^2\equiv1\pmod{n},\quad\text{but}\quad r\not\equiv1\pmod{n}.
\]
Define a group $G$ by the presentation
\[
G=\langle a,b\mid a^m=b^n=1,b^a=b^r\rangle.
\]
Then $G$ is metacyclic $(m,n)$-bicyclic group by Proposition~\ref{Holder}.
 
 Now let $a_1=a^{m_1}$ and $b_1=b$. Then $|a_1|=2^e$, $|b_1|=n$, and $\gcd(|a_1|,|b_1|)=1$. The 
 relation $b^a=b^r$ implies: $b_1^{a_1}=b^{a^{m_1}}=b^{r^{m_1}}$. Since $m_1$ is odd and $r\not\equiv1\pmod{n}$,
 we have $r^{m_1}\not\equiv1\pmod{n}$, so $[a_1,b_1]\neq1$. Hence, $G$ is non-nilpotent by Proposition~\ref{Nilp}.

If instead $n$ is even and $m>1$ is odd, the proof is similar by symmetry.
\end{proof}

\begin{lemma}\label{Action}
If $G$ is an $(m,n)$-bicyclic group containing a normal 
  Sylow $p$-subgroup $P$ such that $\gcd(p-1,|G/P|)=1$, then $G=P\times Q$ where $Q$ is a complement
  of $P$ in $G$.
\end{lemma}
\begin{proof}
By hypothesis, we may assume $G=\langle a\rangle\langle b\rangle$, where $|a|=m$ and 
$|b|=n$, and $1\neq P\lhd G$. Then the conjugation action of $G$ on $P$ defines a homomorphism: 
\[
\alpha:G\to\Aut(P), g\mapsto \alpha_g, \quad\text{where} \quad x^{\alpha_g} =x^g~\text{for all}~ x\in P.
\] 
Since $\Phi(P)\ch P$, we obtain an induced action on $P/\Phi(P)$ defined by
\[
(x\Phi(P))^{\bar\sigma}=x^{\sigma}\Phi(P),\quad \forall x\in P
\]
This gives a homomorphism:
\[
\beta:\Aut(P)\to\Aut(P/\Phi(P)), \sigma\mapsto\bar\sigma.
\]
Let $\gamma:=\alpha\circ\beta$, giving is a homomorphism $\gamma: G\to\Aut(P/\Phi(P))$.

Since $G$ is bicyclic, by Proposition~\ref{Prod}, so is $P$. Thus, $P$ is a $p$-group with $\rank(P)\leq 2$.
By Burnside's Basis theorem, $P/\Phi(P)\cong\mathbb{Z}_p^r$ with $r\leq 2$

\begin{case}[1] $\rank(P)=1$. 

Then $P$ is cyclic, and $\Aut(P)\cong\mathbb{Z}_{p^{s}}^*$ has order $p^{s-1}(p-1)$ for some $s\geq1$. 
Since $Q^{\alpha}\leq \Aut(P)$, so $|Q^{\alpha}|$ divides $p^s(p-1)$. 
Since $\gcd(|Q|,p(p-1))=1$, we have $Q^{\alpha}=1$, i.e., $Q$ centralizes $P$. Hence, $G=P\times Q$.
\end{case}

\begin{case}[2] $\rank(P)=2$.

By Burnside's basis theorem, $P/\Phi(P)\cong\mathbb{Z}_p^2$, which can be viewed as a  
$2$-dimensional $\GF(p)$-vector space $V$ with a bais $(\bar a,\bar b)$ where $\bar a:=a\Phi(P)$ and $\bar b:=a\Phi(P)$.
Then the images  $a^{\gamma}$ and $b^{\gamma}$ are represented by invertible linear transformations over $V$,
lying in the subgroups $U$ and $V$ of $\GL(2,p)$, respectively, defined as follows:
\[
U:=\Big\{
\begin{pmatrix}
1&0\\
i&j
\end{pmatrix}
\mid i\in\GF(p), j\in\GF(p)^*
\Big\}
\]
and
\[
V:=\Big\{
\begin{pmatrix}
k&l\\
0&1
\end{pmatrix}
\mid l\in\GF(p), k\in\GF(p)^*
\Big\}.
\]
Both $U$ and $V$ have order $p(p-1)$, so the group $G^{\gamma}$ has order dividing $p^2(p-1)^2. $ Since $\ker\beta$
is a $p$ group by Proposition~\ref{Basis}(d), $G^{\alpha}=(G^{\gamma})^{\beta^{-1}}$ has order dividing $p^s(p-1)^2$ for
some $s\geq2$. Thus, $|Q^\alpha|$  divides $p^s(p-1)^2$, but $\gcd(|Q|,p(p-1))=1$, so $Q^{\alpha}=1$, and $Q\leq C_G(P)$.
Therefore, $G=P\times Q$
\end{case}
This completes the proof.
\end{proof}

We are ready to prove the main result of the paper.
\begin{proof}[Proof of Theorem~\ref{main}]
If the integers $m$ and $n$ do not satisfy the conditions in~\eqref{Mcond}, then by Lemma~\ref{Rem}--\ref{OddEven} 
 there exists a non-nilpotent $(m,n)$-bicyclic group.

Suppose that $m$ and $n$ satisfy the conditions in~\eqref{Mcond}, and let $G$ be an $(m,n)$-bicyclic group. 
Then, by Corollary~\ref{Rem1}, one of the following holds: 
 \begin{enumerate}[\rm(a)]
 \item $m=1$ or $n=1$,
 \item $m=p^e$ and $n=p^f$ for some prime $p$ and integers $e,f\geq1$,
  \item both $m$ and $n$ are odd numbers, and for each pair of primes $p\neq q$ with $p|m$ and $q|n$, 
  we have $p\nmid (q-1)$ and $q\nmid( p-1).$
 \end{enumerate}

If (a) holds, then $G$ is cyclic.  If (b) holds, then $G$ is a  $p$-group. In both cases, $G$ is certainly is nilpotent. 

Suppose now that (c) holds.  By Proposition~\ref{Bicyclic}, the group 
$G$ is supersolvable, and hence, by Proposition~\ref{Super}, it has a Sylow tower. 
In particular, let $p$ be the largest prime divisor of $|G|$, and let $P\in\Syl_p(G)$. Then $P\lhd G$.  
By the Schur-Zassenhauss theorem, $P$ has a complement $Q$ in $G$, and $Q$ is an $(m_1,n_1)$-bicyclic group
for some positive integers $m_1|m$ and $n_1|n$. It is clear that $m_1$ and $n_1$ inherit the same 
structural conditions
from $m$ and $n$, so the inductive hypothesis applies and  $Q$ nilpotent. 
Applying Lemma~\ref{Action} we obtain $G=P\times Q$. 
Therefore, $G$ is also nilpotent, as required.
\end{proof}

\begin{corollary}\label{Abel}
Every $(m,n)$-bicyclic group is abelian if and only if 
\begin{align}\label{Mcond2}
\gcd(m,\phi(n))=\gcd(n,\phi(m))=1.
\end{align}
\end{corollary}
\begin{proof}
Note that if $n=\prod_{p|n}p^e$, then 
\[
\phi(n)=\prod_{p|n}p^{e-1}(p-1)=\frac{n}{\rad(n)}\phi(\rad(n)).
\]
Suppose first that every $(m,n)$-bicyclic group is abelian. Then, in particular,  every such group is nilpotent.
By Theorem~\ref{main}, it follows that
\[
\gcd(m,\phi(\rad(n)))=\gcd(n,\phi(\rad(m)))=1.
\]
Now assume that $\gcd(m,\phi(n))\neq1$. Then there exists a common prime divisor $p$  such that 
$p^{e}||n$ and $p^f||m$ for some $e\geq2$ and $f\geq1$.
It is possible to construct a non-abelian $(p^f,p^e)$-bicyclic metacyclic group $P$. 
Taking the direct product of $P$ with an abelian $(m/p^f,n/p^e)$-bicyclic group  gives 
 a non-abelian $(m,n)$-bicyclic group, contradicting our assumption.
A similar contradiction arises if $\gcd(n,\phi(m))\neq1$, so condition~\eqref{Mcond2} is necessary.

Conversely, suppose that $\gcd(m,\phi(n))=\gcd(n,\phi(m))=1.$ 
Then it follows that $\gcd(m,\phi(\rad(n)))=\gcd(n,\phi(\rad(m)))=1,$ 
and hence every $(m,n)$-bicyclic group $G$ is nilpotent by the main theorem. 

To show that such a group $G$ is abelian,
it suffices to show that each Sylow $p$-subgroup $P$ of $G$ is abelian.
If $p^e||m$ and $p^f||n$, the assumption implies that $e\leq1$ and $f\leq 1$.
so $|P|\leq p^2$, which implies that $P$ is abelian. Therefore, $G$ is abelian.
\end{proof}

\begin{corollary}
Every $(m,n)$-bicyclic group is cyclic if and only if 
\begin{align}\label{Mcond3}
\gcd(m,n)= \gcd(m,\phi(n))=\gcd(n,\phi(m))=1.
\end{align}
\end{corollary}
\begin{proof}
Suppose every $(m,n)$-bicyclic group is cyclic. Then each such group is in particular abelian. 
By Corollary~\ref{Abel}, we have $\gcd(m,\phi(n))=\gcd(n,\phi(m))=1$.
If $\gcd(m,n)\neq1$, then the abelian group $C_m\times C_n$ is a noncyclic $(m,n)$-bicyclic group,
contradicting our assumption. 

Conversely, suppose that the three conditions in  \eqref{Mcond3} hold.
Then by Corollary~\ref{Abel},
every $(m,n)$-bicyclic group $G$ is abelian. Since $\gcd(m,n)=1$, the group $G\cong C_m\times C_n$ is a direct product, 
so it is cyclic, as required.
\end{proof}
\section{Further problems}
As discussed earlier, there exist classical results concerning
 group factorizations. For instance, It\^{o} showed that the product of two abelian groups is always metabelian,
and Kegel and Wielandt proved that the product of two nilpotent groups is always solvable. 
Building on these foundational results, and in light of the structural classification obtained in this paper, it
is natural to consider the following broader questions:
\begin{problem}
Determine necessary and sufficient conditions on positive integers $m$ and $n$ such that every product of 
two abelian groups of orders $m$ and $n$ is 
\begin{enumerate}[(a)]
\item abelian, 
\item nilpotent, 
\item supersolvable, 
\item or satisfies other structural properties.
\end{enumerate}
\end{problem} 

\begin{problem}
Determine necessary and sufficient conditions on positive integers $m$ and $n$ such that every product of 
two nilpotent groups of orders $m$ and $n$ is 
\begin{enumerate}[(a)]
\item nilpotent, 
\item metabelian,
\item supersolvable, 
\item or has some other specific structural properties.
\end{enumerate} 
\end{problem} 

To provide further context, recall that a positive integer $n$ is cyclic if every group of order $n$ is cyclic; 
this occurs if and only if $\gcd(n,\phi(n))=1$. Similarly, a pair of positive integers $(m,n)$ 
is singular if $\gcd(m,\phi(n))=1$ and $\gcd(n,\phi(m))=1$,
in which case every $(m,n)$-bicyclic group is abelian. 

 Erd\H{o}s proved that the number of cyclic integers $n\leq x$ is asymptotic to 
\[
z(x)=e^{\gamma}\frac{x}{\log\log\log x},\quad \text{as}~x\to\infty,
\]
where $\gamma$ is Euler's constant~\cite{Erdos1948}. This was extended by Nedela and Pomerance,
who showed that the number of singular pairs  with $m,n\leq x$ is asymptotic to $z(x)^2$ as $x\to\infty$~\cite{NP2018}. 

In light of these number-theoretic results, the following question naturally arises:
\begin{problem}
Find an asymptotic formula for the number of integer pairs $(m,n)$ with $m,n\leq x$ such that 
$\gcd(m,\phi(\rad(n))=1$ and $\gcd(n,\phi(\rad(m))=1$.
\end{problem}

\end{document}